\documentclass[10pt,a4paper]{amsart}
\usepackage{hyperref}
\usepackage{mathrsfs}
\usepackage{amsmath}
\usepackage{amssymb}
\usepackage[all]{xy}
\usepackage{latexsym}
\usepackage[T1]{fontenc}
\usepackage{tikz}
\usepackage[utf8]{inputenc}
\title[Connective $KO$-theory]{On connective $KO$-theory of
elementary abelian
$2$-groups}
\author[Geoffrey Powell]{Geoffrey Powell}
\address{Laboratoire Angevin de Recherche en Mathématiques, UMR 6093, 
  Faculté des Sciences, Université d'Angers, 
2 Boulevard Lavoisier,
49045 Angers, France}
\email{Geoffrey.Powell@math.cnrs.fr}
\date{}
\thanks{{\em Acknowledgement:} 
The author is extremely grateful to Bob Bruner for many
conversations related to this material during his visit to the Université Paris
13 in May 2012 as {\em  Professeur invité} and for subsequent discussions.}

\newtheorem{thm}{Theorem}[section]
\newtheorem{prop}[thm]{Proposition}
\newtheorem{cor}[thm]{Corollary}
\newtheorem{lem}[thm]{Lemma}

\theoremstyle{definition}
\newtheorem{defn}[thm]{Definition}
\newtheorem{exam}[thm]{Example}

\theoremstyle{remark}
\newtheorem{rem}[thm]{Remark}
\newtheorem{nota}[thm]{Notation}
\newcommand{\vd}{V^\sharp}
\newcommand{\bock}{\mathfrak{B}}

\newcommand{\annih}{\mathrm{Ann}}

\newcommand{\Pbarzedtwo}{\overline{P}_{\zed_2}}

\newcommand{\f}{\mathscr{F}}
\newcommand{\obj}{\mathrm{Ob}\hspace{1pt}}

\newcommand{\nat}{\mathbb{N}}
\newcommand{\zed}{\mathbb{Z}}

\newcommand{\field}{\mathbb{F}}
\newcommand{\cala}{\mathcal{A}}
\newcommand{\Ibar}{\overline{I}_\field}
\newcommand{\Pbar}{\overline{P}_\field}

\newcommand{\dcouple}{D}
\newcommand{\ecouple}{E}
\newcommand{\fundcx}[1][\bullet]{\mathscr{E}_{#1}}

\newcommand{\koimage}{QO}
\renewcommand{\hom}{\mathrm{Hom}}
\renewcommand{\phi}{\varphi}
\renewcommand{\epsilon}{\varepsilon}
\begin{document}
\begin{abstract}
A general notion of detection is introduced and used in
the study of the cohomology of elementary abelian $2$-groups with respect to the
spectra in the Postnikov tower of orthogonal $K$-theory. This recovers and
extends results of Bruner and Greenlees and is related to calculations of the
(co)homology of the spaces of the associated $\Omega$-spectra by Stong and by
Cowen Morton.
\end{abstract}
\maketitle

\section{Introduction}
\label{sect:intro}

The orthogonal $K$-theory of elementary abelian $2$-groups possesses a rich
structure and the spectra of the Postnikov tower of $KO$ lead to interesting
related 
 functors 
$
 V \mapsto 
KO \langle n \rangle ^* (BV). 
$
The study of these  is, for example, a first step towards a systematic analysis
of $KO
\langle n \rangle ^*
(BG)$, for finite groups $G$. Bott periodicity reduces to consideration of $ko=
KO
\langle 0 \rangle$, $ko\langle 1 \rangle$, $ko \langle 2 \rangle$ and
$ko\langle 4 \rangle$, of which the case $ko$ has been studied extensively (but
non-functorially) by Bruner and Greenlees \cite{bg2}, based on their
earlier work on the complex case \cite{bruner_greenlees}. A key property is that
$ko^* (BV)$ is {\em detected} by the periodic theory $KO^*
(BV)$ together with integral cohomology $H\zed^*(BV)$, via the zero layer $ko
\rightarrow H\zed$ of the Postnikov tower. The main result of the paper (Theorem
\ref{thm:postnikov_detect}) 
establishes the analogous property for the spectra $KO\langle n \rangle$ using
the
Postnikov layers $KO \langle n \rangle \rightarrow \Sigma^n H (KO_n)$; this
leads to a description of $KO \langle n \rangle^* (BV)$ (see Corollary
\ref{cor:ses_kon}). This recovers, in particular, the results of Bruner and
Greenlees \cite{bg2} for $ko$.

The functorial structure gives
information on the spaces of the associated $\Omega$-spectra: 
Lannes' theory (cf. \cite{hls,sch}) implies that $V \mapsto ko \langle n
\rangle^d (BV)$
determines (up to $F$-isomorphism) the mod-$2$ cohomology  of the $d$th space of
the
$\Omega$-spectrum associated to $ko\langle n \rangle$. This
establishes the relation with results in the literature:
the mod-$2$ cohomology rings of the connective covers of the classifying space
$BO$ of
the infinite orthogonal group were determined by Stong \cite{stong} and the Hopf
ring for $ko$ and the Hopf module structures of the spectra $ko \langle n
\rangle$ over this Hopf ring were calculated by Cowen Morton
\cite{cowen_morton}. Both these results establish an analogue of the detection
property. Hence, the detection property for $ko^* (BV)$ is related to the
unstable
origin of the fact that $ko \rightarrow H\zed$ induces a
monomorphism in homology $H\field_* (ko) \hookrightarrow H\field_* H \zed$; a
similar statement holds for $KO\langle n \rangle$.

The main results of the paper (see Section
\ref{sect:postnikov_detect}) give a description of the functors $KO \langle n
\rangle ^* (BV)$, based in part on the author's previous work \cite{powell} on
the case of complex
connective $K$-theory, which revisited the earlier work of Bruner and Greenlees
\cite{bruner_greenlees} from a functorial viewpoint, using new techniques.
The abstract treatment of the detection property (given in
Section \ref{sect:detect})  leads to
an explicit relationship between the part of the theory which is detected in the
periodic theory and the torsion part (see Theorem
\ref{thm:functorial_ses_detection}). These methods also apply to the study of
$KO \langle n \rangle _* (BV)$, for all $n$; this leads to a conceptual
understanding of the
relationship between cohomology and homology via the local cohomology spectral,
generalizing  the results of \cite{powell} for  $ku$. This will be explained
elsewhere.

The proof requires an understanding of the homology of a complex which arises
from the primary $k$-invariants of the Postnikov tower of $KO$, taking the
cohomology of the classifying spaces $BV$ (see Section
\ref{sect:example_ku_ko}); the complex  is derived from an exact complex
$\fundcx$ of $\cala (1)$-modules (recall that $\cala (1)$ is the subalgebra of
the
mod-$2$ Steenrod algebra $\cala$ generated by $Sq^1$ and $Sq^2$), related to the
exact complex of
Toda \cite{toda}. The restriction to the category of $\cala (1)$-modules
provides the tools for
calculating the homology of the cochain complex $\hom_{\cala(1)} (\fundcx ,
H\field^* (BV))$ (see Sections \ref{sect:module} and \ref{sect:calculate}),
based on ideas of Ossa \cite{ossa} 
developed in the thesis of Cherng-Yih Yu \cite{yu} and by Bruner
\cite{bruner_Ossa}.

The first step towards establishing detection is to treat the case of $ko$ (see
Section \ref{sect:cohom}). Much of the argument can be carried out using
detection in periodic complex $K$-theory and the known structure of $ku^* (BV)$.
However, this is not sufficient to treat the classes which are divisible by
$\eta$ and which are detected in $KO$-cohomology; for these a general argument
(cf. Proposition \ref{prop:eta-torsion}) related to the $\eta$-Bockstein
spectral sequence is used, for which  Proposition \ref{prop:sq_cohom} is the
crucial calculational input. 

This leads to the determination of the functor $ko^* (BV)$ (see Corollary
\ref{cor:ko_cohom}); from this, detection is deduced  for   $ko\langle n \rangle
^* (BV)$ in general (Theorem
\ref{thm:postnikov_detect}), whence  the functorial description given
in Corollary \ref{cor:ses_kon}.

\section{Abstract detection}
\label{sect:detect}

Consider a tower $E_\bullet$ over $F$ in the stable homotopy category:

\[
\xymatrix{
\ldots
\ar[r]
& 
E_n
\ar[r]^{e_n}
\ar[rrd]_{f_n}
&
E_{n-1} 
\ar[r]^{e_{n-1}}
\ar[rd] ^{f_ {n-1}}
&
E_{n-2}
\ar[r]
\ar[d]^{f_{n-2}}
&
\ldots\\
&&&F.
}
\]
This has associated cofibre sequences
\[
 \xymatrix{
E_n \ar[r]^{e_n}
&
E_{n-1} 
\ar[r]^{c_{n-1}}
&
C_{n-1} 
\ar[r]^{\delta_{n-1}}
&
\Sigma E_n
}
\]
with composite morphism $\theta_n$  playing the role of a primary
$k$-invariant:
\[
 \xymatrix{
C_n \ar[r]^{\delta_n}
\ar@/_1pc/[rr]_{\theta_n}
&
\Sigma E_{n+1} 
\ar[r]^{\Sigma c_{n+1}}
&
\Sigma C_{n+1}.
}
\]
These fit into the following commutative diagram, in which the
horizontal sequence is the cofibre sequence:
\begin{eqnarray}
 \label{diag:basic}
 \xymatrix{
\Sigma ^{-1}C_{n-1}
 \ar[rd]^{\Sigma^{-1} \theta_{n-1}}
\ar[d]_{\Sigma ^{-1} \delta_{n-1}}
\\
E_n 
\ar[r]_{c_n}
&
C_n 
\ar[rd]_{\theta_n}
\ar[r]^{\delta_n}
&
\Sigma E_{n+1} 
\ar[r]
\ar[d]^{\Sigma c_{n+1}}
&
\Sigma E_n
\\
&&\Sigma C_{n+1}.
}
\end{eqnarray}

\begin{rem}
The composite $\theta_n \circ \Sigma^{-1}\theta_{n-1}$ is
trivial.
\end{rem}

Detection is defined with respect to a fixed object $X$ by
considering the behaviour of $[X, E_n]^*
\rightarrow [X, E_{n-1}]^*$ and $[X, E_n]^* \rightarrow [X, F]^*$, where
$[-,-]^*$ denotes the $\zed$-graded morphism groups.

\begin{defn}
 For $n \in \zed$ and a spectrum $X$, the tower satisfies 
\begin{enumerate}
 \item 
level $n$ detection if 
$
(f_n, c_n) :  [X, E_n]^* \rightarrow 
[X, F]^* \oplus [X, C_n]^*
$
is a monomorphism;
\item
weak level $n$ detection if 
$
(e_n , c_n) : [X, E_n]^* \rightarrow 
[X, E_{n-1}]^* \oplus [X, C_n]^*
$
is a monomorphism.
\end{enumerate}
\end{defn}

\begin{rem}
One can also consider a family of spectra and define detection pointwise; this
reduces to the single object case by taking the coproduct of the family.
\end{rem}

\begin{exam}
\label{exam:detect_BV}
 The case of interest here is the family $\Sigma^\infty BV$, as $V$ ranges over
a skeleton of the category of finite rank elementary abelian $2$-groups.
\end{exam}

\begin{lem}
\label{lem:weak-strong}
 For $n \in \zed$, 
\begin{enumerate}
 \item 
level $n$ detection $\Rightarrow$ weak level $n$ detection;
\item
(level $n-1$ detection and weak level $n$ detection) $\Rightarrow$ 
level $n$ detection.
\end{enumerate}
\end{lem}

\begin{proof}
Straightforward.
\end{proof}

From the construction, it is clear that $c_{n-1}: E_{n-1}\rightarrow C_{n-1}$
induces a morphism
\[
 [X,  E_{n-1}]^* \rightarrow
\mathrm{Ker} \big\{
[X,  C_{n-1}] ^*
\stackrel{\theta_{n-1}}{\longrightarrow}
[X, \Sigma C_n]^*
\big\}.
\]

The following result gives an alternative formulation of weak detection.

\begin{lem}
 \label{lem:weak-equivalent}
For $n \in \zed$, the following conditions are equivalent:
\begin{enumerate}
 \item 
weak level $n$ detection holds;
\item
$c_{n-1}$ induces a surjection $
 [X, E_{n-1}]^* \twoheadrightarrow
\mathrm{Ker}({\theta_{n-1}}).
$
\end{enumerate}
\end{lem}

\begin{proof}
 $(2) \Rightarrow (1)$. Suppose that $x \in [X, E_n]^*$ lies in the kernel of
$(e_n, c_n)$; since $x$ is in the kernel of $e_n$, it is the image of
some  $\tilde{x}\in [X, \Sigma^{-1} C_{n-1}]^*$ and, moreover,
$\tilde{x}$ lies in the kernel of $\Sigma^{-1} \theta_{n-1}$. Hence, by the
hypothesis (2), $\tilde{x}$ is the image of an element of $[X, \Sigma^{-1}
E_{n-1}]^*$. This implies that $\tilde{x} \mapsto 0$ in $[X, E_n]^*$, so that
$x=0$, thus 
weak level $n$ detection holds. 

 $(1) \Rightarrow (2)$. Consider an element $y \in [X, \Sigma^{-1} C_{n-1}]^*$
which lies in the kernel of $\Sigma^{-1} \theta_{n-1}$ and set $\overline{y}:= 
\Sigma^{-1} \delta_{n-1} y \in [X, E_n]^*$. 
 Since $e_n \Sigma^{-1} \delta_{n-1} =0$, $e_n\overline{y} =0$ and the 
hypothesis on $y$ implies that $c_n \overline{y}= \Sigma^{-1} \theta_{n-1} y = 
0$. 
Hence, weak detection implies that  $\overline{y} \in [X, E_n]^*$ is zero; by 
exactness, $y$ is the image of a class
in $[X, \Sigma^{-1}E_{n-1}]^*$, as required. 
\end{proof}

\begin{nota}
 For $n\in \zed$, write $\Phi_n [X, F]^*$ for the image of $[X, E_n]^*
\stackrel{f_n}{\rightarrow}[X, F]^*$.
\end{nota}

This gives the decreasing filtration:
\[
\ldots \subset \Phi_n [X, F]^* \subset \Phi_{n-1}[X, F]^* \subset \ldots
\subset
[X, F]^*.
\]

\begin{lem}
\label{lem:image_surject}
 For $n\in \zed$, $f_{n-1}$ induces a surjection 
\[
 \mathrm{Image} \big \{
[X, E_n]^* \stackrel{e_n}{\rightarrow}
[X, E_{n-1}]^* 
\big\}
\twoheadrightarrow
\Phi_n [X, F]^*.
\]
If level $n-1$ detection holds, then this is an isomorphism. 
\end{lem}

\begin{proof}
 The first statement is clear, since $f_n=f_{n-1}\circ e_n$. The second
statement is a  consequence of the fact that the composite 
\[
 [X, E_n]^* \stackrel{e_n}{\rightarrow} [X, E_{n-1}]^*
\stackrel{c_{n-1}}{\rightarrow} [X, C_{n-1}]^* 
\]
is trivial, together with the hypothesis that level $n-1$ detection holds. 
\end{proof}

\begin{prop}
\label{prop:detection_subquotient}
 For $n \in \zed$, there are natural morphisms
\[
 \xymatrix{
\mathrm{Ker} (\delta_n)/ \mathrm{Im} (\Sigma^{-1} \theta_{n-1})
\ar@{^(->}[r]^{\iota_n}
\ar[d]^\cong
&
\mathrm{Ker}(\theta_n ) / \mathrm{Im} (\Sigma^{-1} \theta_{n-1})
\\
\mathrm{Im} (e_n) / \mathrm{Im} (e_n\circ e_{n+1})
\ar@{->>}[r]_{\sigma_n}
&
\Phi_n [X, F]^* / \Phi_{n+1} [X, F]^*.
}
\]
In particular, $\Phi_n [X, F]^* / \Phi_{n+1} [X, F]^*$ is a subquotient of
$\mathrm{Ker}(\theta_n ) / \mathrm{Im} (\Sigma^{-1} \theta_{n-1} )$.

Moreover, 
\begin{enumerate}
\item
weak level $n+1$ detection holds if and only if $\iota_n$ is an isomorphism;
 \item
if level $n-1$ detection holds, then $\sigma_n$ is an isomorphism. 
\end{enumerate}
If both the above conditions hold, then 
 \[
 \mathrm{Ker}(\theta_n) / \mathrm{Im} (\Sigma^{-1} \theta_{n-1} )
 \cong
  \Phi_n [X, F]^* / \Phi_{n+1} [X, F]^*
.
 \]
\end{prop}

\begin{proof}
From diagram (\ref{diag:basic}), there are inclusions:
$ 
 \mathrm{Im}(\Sigma^{-1} \theta_{n-1}) 
\subset 
\mathrm{Im}(c_{n})
= 
\mathrm{Ker}(\delta_n)
\subset \mathrm{Ker} (\theta_n).
$
 The inclusion $\iota_n$ is induced by  $\mathrm{Ker}(\delta_n)
\subset \mathrm{Ker} (\theta_n)$ and the equivalence between 
weak level $n+1$ detection and $\iota_n$ being an isomorphism follows from Lemma
\ref{lem:weak-equivalent}.

The surjection $\sigma_n$ is given by Lemma \ref{lem:image_surject}, using the
argument outlined in the proof of {\em loc. cit.} to show that this is an
isomorphism under the hypothesis of level $n-1$ detection.

Using the equality $\mathrm{Im}(c_{n}) = \mathrm{Ker} (\delta_n)$, the
vertical morphism is induced by $e_n$, which gives a well-defined surjection:
\begin{eqnarray}
\label{eqn:surject_deltan}
 \mathrm{Ker}(\delta_n) \twoheadrightarrow \mathrm{Im} (e_n) / \mathrm{Im}
(e_n\circ e_{n+1}).
\end{eqnarray}
The cofibre sequence $\Sigma^{-1} C_{n-1} \stackrel{\Sigma^{-1}
\delta_{n-1}}{\longrightarrow} E_n
\stackrel{e_n}{\rightarrow} E_{n-1}$ induces an exact sequence 
\[
 [X, \Sigma^{-1} C_{n-1} ]^* 
\rightarrow 
[X, E_n]^* 
\rightarrow 
[X,E_{n-1}]^*,
\]
and it is straightforward to deduce  that the kernel of the surjection
(\ref{eqn:surject_deltan}) 
is the image of $\Sigma^{-1} \theta_{n-1}$, as required.
\end{proof}

\begin{thm}
\label{thm:functorial_ses_detection}
	Suppose that detection holds  $\forall n \in \zed$, then there are short
exact sequences (natural in $\mathrm{End}(X)$)
\[
	0
	\rightarrow
	\mathrm{Im} (\Sigma^{-1}\theta_{n-1})
	\rightarrow
	[X, E_n]^*
	\rightarrow
	\Phi_n [X, F]^*
	\rightarrow
	0
\]
which are formed  by  pullback along the natural surjection $$\Phi_n [X, F]^*
\twoheadrightarrow \Phi_n[X, F]^* / \Phi_{n+1} [X, F]^*$$ of the short exact
sequence:
\[
	0
	\rightarrow
	\mathrm{Im} (\Sigma^{-1}\theta_{n-1})
	\rightarrow
	\mathrm{Ker} ( \theta_{n})
	\rightarrow
	\Phi_n [X, F]^*/ \Phi_{n+1} [X, F]^*
	\rightarrow
	0.
\]
\end{thm}

\begin{proof}
	By definition, $f_n$ induces a surjection $[X, E_n]^*
\twoheadrightarrow \Phi_n [X, F]^*$. Since level $n-1$ detection holds, the
kernel coincides with the kernel of $[X, E_n]^* \stackrel{e_n}{\rightarrow} [X,
E_{n-1}]^*$ (as in the proof of Lemma \ref{lem:image_surject}) and hence
identifies with the image of
	\[
		[X, \Sigma^{-1} C_{n-1}] ^* 
\stackrel{\Sigma^{-1}\delta_{n-1}}{\rightarrow} [X, E_n]^*.
	\]
By level $n$ detection, this image is detected in $[X, C_n]^*$, where it
identifies with the image of $\Sigma^{-1} \theta_{n-1}$, by definition of the
latter.

Lemma \ref{lem:weak-equivalent}, using the level $n+1$ detection hypothesis,
implies that $c_n$ induces a surjection  $[X, E_n] \twoheadrightarrow
\mathrm{Ker} (\theta_n)$. Combining this with Proposition
\ref{prop:detection_subquotient} shows that there is a pullback square:
\[
	\xymatrix{
[X, E_n]^*
\ar@{->>}[r]
\ar@{->>}[d]
&
\Phi_n [X, F]^*
\ar@{->>}[d]
\\
\mathrm{Ker}(\theta_n)
\ar@{->>}[r]
&
\Phi_n [X, F]^*/ \Phi_{n+1} [X, F]^*,
	}
\]
level $n$ detection ensuring that $[X, E_n]^*$ embeds into $\Phi_n [X,F]^*
\oplus  \mathrm{Ker}(\theta_n)$. This proves the final statement.
\end{proof}

\section{Functors}
\label{sect:functors}

This section  introduces the categories of functors which
feature in the paper and the objects which occur, using the
notation of  \cite{powell}. Let $\field$ denote the prime field with two
elements and consider the category of functors
 from finite-dimensional $\field$-vector
spaces to abelian groups; this contains the category $\f$ of functors from
finite-dimensional $\field$-vector spaces to $\field$-vector spaces as a full
subcategory.  A functor is finite if it has a finite composition series and
locally finite if it is the colimit of its finite subobjects.

In order to consider only covariant functors,
vector space duality (denoted here by $V \mapsto\vd$) is used where appropriate.

\begin{exam}
\label{exam:HnBV}
 A basic example is provided by the functor $V \mapsto H\field^* (B\vd)$ of
group cohomology with $\field$-coefficients (cohomology is always 
taken to be reduced; where necessary, a disjoint
basepoint $(-)_+$ is added). In degree $n>0$, this identifies with the $n$th
symmetric power functor $S^n$, which is finite.
\end{exam}

\begin{nota}
	Denote by
	\begin{enumerate}
		\item
		$\Pbarzedtwo$  the augmentation ideal of the $\zed_2$-group ring
functor $\zed_2 [V]$;
		\item
		$\Pbar$ the augmentation ideal of the $\field$-group ring
functor $\field [V]$;
		\item
		$\Pbarzedtwo^n$ (respectively $\Pbar^n$) the $n$th power of the
augmentation ideal $\Pbarzedtwo$ (resp. $\Pbar$), which is understood as
$\Pbarzedtwo$ (resp. $\Pbar$) for $0 \geq n \in \zed$;
		\item
		$\Ibar$ the sub-functor of $V \mapsto \field^{\vd}$ of maps
which send $0$ to zero;
		\item
		$p_n \Ibar \subset \Ibar$ the largest subfunctor of $\Ibar$ of
polynomial degree $n$.
		\end{enumerate}
\end{nota}

\begin{rem}
\label{rem:Ibar_Pbar}
\ 
\begin{enumerate}
 \item
$\Ibar$ is locally finite and uniserial; explicitly, $\Ibar =
\lim_\rightarrow p_n \Ibar$ and $p_n
\Ibar$ is finite, uniserial with composition factors $\Lambda^1, \ldots ,
\Lambda^n$, where $\Lambda^j$ is the $j$th exterior power functor,
which is an object of $\f$ and is simple. 
\item
$\Pbar$ is dual to $\Ibar$ and hence is uniserial and {\em not}
locally finite (for duality, see
\cite{kuhn1,powell}); the filtration by
powers of the augmentation
induces short exact sequences 
$
 0 
\rightarrow 
\Pbar^{n+1} \rightarrow \Pbar^n 
\rightarrow 
\Lambda^n
\rightarrow 
0
$, 
for $0< n \in \zed$.
\end{enumerate}
\end{rem}

\begin{nota}
\label{nota:grothendieck}
 Let $F, G$ be finite functors. 
\begin{enumerate}
 \item 
Write $[F]$ for the element of the Grothendieck group of finite
functors corresponding to $F$, so that $[F]= \sum_\lambda a_\lambda
[S_\lambda]$, where 
$a_\lambda \in \nat$ is the multiplicity of the simple  $S_\lambda$ in
$F$;  the function $a_{(-)}$ has finite support and the graded associated
to a composition series of $F$ is $\mathrm{gr} (F) \cong \bigoplus_\lambda 
S_\lambda^{\oplus a_\lambda}$.
\item
Write $[F] \leq [G]$ if
the associated graded $\mathrm{gr}(F)$ of a composition series of $F$ is a
direct summand of $\mathrm{gr}(G)$.
(This can be interpreted as an inequality of multiplicities of composition
factors.)
\end{enumerate}
\end{nota}

\begin{exam}
\label{exam:grothendieck_gp}
For $t\in \nat$, there are equalities in the Grothendieck group: 
\begin{enumerate}
 \item 
$
 [p_t \Ibar] = \sum_{j=1}^t [\Lambda^j].
$ 
\item
$
[\Pbar / \Pbar^{t+1}] = [p_t \Ibar].
$
\end{enumerate}

\end{exam}

The following is clear:

\begin{lem}
\label{lem:subquotient_Grothendieck}
If $F$ is a subquotient of a finite functor $G$, then $[F] \leq [G]$.
\end{lem}

The following result gives information on the filtration by powers
of the augmentation ideal of $\Pbarzedtwo$.

\begin{prop}
\cite{powell}
\label{prop:Grothendieck-group}
For $n \in \nat$, the canonical inclusion $\Pbarzedtwo^{n+1} \hookrightarrow
\Pbarzedtwo^n$ induces a short exact sequence 
	$
		0\rightarrow \Pbarzedtwo^{n+1} \hookrightarrow \Pbarzedtwo^n
\rightarrow p_n \Ibar \rightarrow 0.
	$ In particular, the cokernel of the inclusion $\Pbarzedtwo^{n+1}
\hookrightarrow
\Pbarzedtwo$ is a finite functor and
$
	[\Pbarzedtwo / \Pbarzedtwo^{n+1} ] = \sum_{j=1} ^n [p_j \Ibar].
$
\end{prop}

The $2$-adic filtration of $\Pbarzedtwo$ and its  relationship with the
filtration by powers of the augmentation ideal is  of importance; 
there is a short exact sequence
$
 0
\rightarrow
\Pbarzedtwo
\stackrel{2}{\rightarrow}
\Pbarzedtwo
\rightarrow
\Pbar
\rightarrow
0
$
which restricts (for $n >0$) to the short exact sequence:
$
 0
\rightarrow
\Pbarzedtwo^{n}
\stackrel{2}{\rightarrow}
\Pbarzedtwo^{n+1}
\rightarrow
\Pbar^{n+1}
\rightarrow
0.
$

\begin{figure}[h!]
\begin{tikzpicture}
\path[fill=gray!30]
(0,1)  --(3,1) --(5,3)--(5,5)--(0,5)  --cycle;
\path[draw]
(0,0)--(5,0)--(5,3);
\path[draw,style=dashed]
(1,0) --(5,4);
\path[draw,style=dotted]
(0,1) --(5,1);
\path[draw,line width=2pt]
(0,0)--(2,0)  --(5,3) --(5,5) ;
\node at (2,.5) {$\Lambda^{n+1}$};
\node at (4,2.5) {$p_n\Ibar$};
\node at (.5,.5) {$\Pbar^{n+2}$};
\end{tikzpicture}
\caption{A representation of the  subfunctors   $\Pbarzedtwo^n \subset
\Pbarzedtwo^{n+1} \subset\Pbarzedtwo$}
\label{figure:Pbarzed}
\end{figure}

This is illustrated by Figure \ref{figure:Pbarzed}, in which the bounding
square 
represents $\Pbarzedtwo$, the subfunctor $\Pbarzedtwo^{n+1}$ is
bounded by the heavy line and the shaded region indicates the subfunctor
$2 \Pbarzedtwo^n \subset \Pbarzedtwo^{n+1}$, which is isomorphic to 
$\Pbarzedtwo^n$. The region above the  dotted line represents the inclusion $2
\Pbarzedtwo
\subset \Pbarzedtwo$,
whereas the region above the dashed line represents the inclusion
$\Pbarzedtwo^{n+2} \subset
\Pbarzedtwo^{n+1}$, which restricts in the shaded region to the inclusion
$\Pbarzedtwo^{n+1}\subset \Pbarzedtwo^n$. The indicated functors represent the
subquotients corresponding to the respective areas. Hence the bottom row
corresponds to the exact sequence
$
 0
\rightarrow 
\Pbar^{n+2}
\rightarrow 
\Pbar^{n+1} 
\rightarrow 
\Lambda^{n+1}
\rightarrow
0
$
and the diagonal to 
$
  0
\rightarrow 
p_n\Ibar
\rightarrow 
p_{n+1}\Ibar 
\rightarrow 
\Lambda^{n+1}
\rightarrow
0.
$

\begin{defn}
For $F \in \obj \f$ taking finite-dimensional values, the Poincaré series $p_F$
is $$
		p_F (t) := \sum_{i\geq 0} \dim F(\field^i) t^i.
	$$
\end{defn}

The following general result concerning functors of $\f$ (taking values
in $\field$-vector spaces) is used in Section \ref{sect:calculate} to deduce
functorial information from Poincaré series. 

\begin{lem}
\label{lem:exterior_powers}
	Let $F \in \obj \f$ be  finite and suppose that $p_F (t)= \sum_{i=0}
^\infty \epsilon_i {t \choose i}$, with $\epsilon_i \in \{0,1 \}$, then
$\epsilon_i$ has finite support and 
	$
		[F] = \sum _{i=0} ^\infty \epsilon_i [\Lambda^i].
	$
\end{lem}

\begin{proof}
The Poincaré series $p_F$ only depends upon $[F]$,  hence the result is
a consequence of the fact that, for each natural number $n$, there is a
unique simple functor $S$ in $\f$ such that $S(\field^i)$ is trivial for $i < n$
and
$\dim S (\field^n )=1$, namely the exterior power functor $\Lambda^n$, together
with the fact that $\dim \Lambda^n (\field^d) = {n \choose d}$. The finiteness
hypothesis on $F$ clearly implies that $\epsilon_i$ has finite support.
\end{proof}

\section{Background on the spectra associated to $K$-theory} 
\label{sect:example_ku_ko}

\subsection{The tower associated to $KU$-theory}
As usual, $ku$ is written for  $KU\langle 0 \rangle$ and  Bott periodicity gives
the isomorphisms $  KU\langle 2n
\rangle \cong \Sigma^{2n} ku $ and $KU\langle 2n+1 \rangle \cong KU\langle 2n+2
\rangle$, for $n \in \zed$, so that the associated cofibre sequences (as in
Section \ref{sect:detect}) are determined  by
\[
	\Sigma^2 ku \stackrel{v}{\rightarrow} ku \rightarrow H \zed \rightarrow
\Sigma^3 ku ,
\]
where $v$ is multiplication by the Bott element, where $KU_* \cong \zed [v^{\pm
1}]$.

The functorial description given in \cite{powell} 
is a consequence of the fact that detection holds in the Postnikov tower of
$KU$: the morphisms $ku \rightarrow KU$ and $ku \rightarrow
H\zed$ induce a monomorphism
$
	ku^* (B\vd) \hookrightarrow H\zed^* (B\vd) \oplus KU^* (B\vd).
$
(This property was observed  by Bruner and Greenlees in
\cite{bruner_greenlees}.) 
 Integral cohomology $H\zed^* (B\vd)$ embeds in $H\field^*
(B\vd)$ as the kernel of the Bockstein, hence there is a monomorphism $ku^*
(B\vd) \hookrightarrow
H\field^* (B\vd) \oplus KU^* (B\vd)$. The structure of these functors can be
described explicitly. 

\begin{nota}
\label{nota:Q_TU}
 (Cf. \cite{bg2}.)
Let $Q^*$ (respectively $TU^*$) denote the image (resp. kernel) of $ku^* (B\vd)
\rightarrow KU^* (B\vd)$.
\end{nota}

Recall that  the Milnor derivations $Q_0, Q_1$ are given by  $Q_0 = Sq^1$, $Q_1
=
[Sq^2, Sq^1]$.

\begin{thm}
\label{thm:ku-detection}
\cite{powell}
Detection holds for the Postnikov tower of $KU$ at all levels. In
particular, there is a natural short exact sequence:
	\[
		0
		\rightarrow
		TU^*
		\rightarrow
		ku^* (B\vd)
		\rightarrow
		Q^* \rightarrow
		0,
	\]
where
\[
 Q^n \cong
\left\{
\begin{array}{ll}
0 & n \mathrm{\  odd}\\
\Pbarzedtwo^d & n= 2d \geq 0\\
\Pbarzedtwo& n =2d \leq 0
\end{array}
\right.
\]
and $TU^*$ identifies with the
image of $Q_0 Q_1 : H\field^{*-4} (B\vd) \rightarrow
H\field^{*}(B\vd)$.
\end{thm}

\subsection{The tower associated to  $KO$-theory}
\label{subsect:example_ko}

Recall that 
$$
	KO_* \cong
\zed [\eta,\alpha, \beta^{\pm 1} ] / (2 \eta,\eta^3,\eta\alpha,\alpha^2 -
4\beta),
$$
 where $|\eta |=1$, $|\alpha|=4$ and $\beta$ is the Bott element, with
$|\beta|=8$. Bott periodicity gives 
$
 KO \langle n + 8 \rangle \cong \Sigma ^8 KO \langle n \rangle
$
for $n\in \zed$; the spectrum $KO\langle 0 \rangle$  is
denoted $ko$.

The Postnikov tower for $KO$ can be deduced by Bott periodicity from:

\begin{equation}
\label{diag:ko_post}
\xymatrix{
\Sigma^8 ko 
\ar[r]^\simeq 
&
ko \langle 8 \rangle 
\ar[d]
\ar[r]
&
\Sigma ^8 H \zed 
\\
&
ko \langle 4 \rangle 
\ar[d]
\ar[r]
&
\Sigma ^4 H \zed
\ar@{-->}[ul]
\ar@/_1pc/@{.>}[u]_{Sq^5}
\\
&
ko \langle 2 \rangle 
\ar[d]
\ar[r]
&
\Sigma ^2 H \field
\ar@{-->}[ul]
\ar@/_1pc/@{.>}[u]_{Sq^3}
\\
&
ko \langle 1 \rangle 
\ar[d]
\ar[r]
&
\Sigma ^1 H \field
\ar@{-->}[ul]
\ar@/_1pc/@{.>}[u]_{Sq^2}
\\
&
ko 
\ar[r]
&
H \zed.
\ar@{-->}[ul]
\ar@/_1pc/@{.>}[u]_{Sq^2}
\\
}
\end{equation}
(The dashed and curved arrows have the usual  degree shift.) The curved arrows
are the associated $k$-invariants;  the cohomology operations are interpreted as
in \cite[Section A.5]{bg2} (see Remark \ref{rem:Sq5} below). 

The associated diagram in mod-$2$ singular cohomology is well understood (cf.
\cite[Section A.5]{bg2} or \cite{adams_priddy}); in particular, $H\field_2^* (ko
\langle n \rangle)$ is a cyclic module over the mod-$2$ Steenrod algebra $\cala$
and the morphism in cohomology induced by $ko \langle n \rangle \rightarrow
\Sigma ^n H (KO_n)$ is surjective. It follows that the curved arrows 
induce a periodic, exact sequence of $\cala$-modules; this is the key exact
sequence of Toda \cite{toda}  used by Stong in \cite{stong}.

\begin{rem}\label{rem:Sq5}
The operation denoted $Sq^5$ in (\ref{diag:ko_post}) is an integral lift of 
$Sq^2 Sq^1 Sq^2$.
 The equivalence of the two descriptions follows from the  Adem relation
$Sq^5 = Sq^2 Sq^1 Sq^2 + Sq^4 Sq^1$, since $Sq^4Sq^1$ lifts trivially to
integral cohomology. 
\end{rem}

\begin{nota}
Recall that $\cala (1)$ (respectively $\cala (0)$) is the finite sub-Hopf
algebra of $\cala$  generated by $Sq^1, Sq^2$ (respectively $Sq^1$) and $\cala
(1) //\cala (0)$ is the induced $\cala(1)$-module $\cala (1)
\otimes_{\cala (0)} \field$.
\end{nota}

The Toda exact complex is induced from an exact complex  $\fundcx$ of
$\cala(1)$-modules by applying the induction functor $\cala \otimes_{\cala (1)}
- $. The complex $\fundcx$ is the periodic extension of:

\begin{equation}
\label{eqn:fundcx}
\xymatrix{
 & 
\Sigma^{-4}(\cala(1) // \cala (0) )
\ar[l]
\\
&\cala (1)//\cala(0) 
\ar[u]_{Sq^2 Sq^1 Sq^2}
&
\Sigma^1 \cala (1)
\ar[l]_{Sq^2}
&
\Sigma ^2 \cala (1)
\ar[l]_{Sq^{2}} 
&
\Sigma^4 (\cala(1) //\cala(0))
\ar[l]_{Sq^3}
\\
&&&&
\Sigma^{8}(\cala (1) //\cala (0) )
\ar[u]^{Sq^2Sq^1Sq^2}
&
,
\ar[l]
}
\end{equation}
in which each morphism is of degree $1$. On the level of objects
$\fundcx[0]=\cala (1)//\cala(0)  $, $\fundcx [1] = \Sigma^1 \cala (1)$,
$\fundcx[2]= \Sigma ^2 \cala (1)$ and $\fundcx [3] = \Sigma^4 (\cala(1)
//\cala(0))$;  for $0 \leq i \leq 3$ and $k \in \zed$,
$\fundcx[4k+i]=\Sigma^{8k} \fundcx[i]$.

For an $\cala(1)$-module $M$, $\hom_{\cala (1)} (\fundcx, M) $ is a periodic (up
to suspension) cochain complex of $\field$-vector spaces, which is of the
following form:

\begin{equation}
\xymatrix{
\ar[r]&
\Sigma^{4} \annih_{(Sq^1)}M 
\ar[d]_{Sq^2 Sq^1 Sq^2}
\\ 
 &\annih_{(Sq^1)}M 
\ar[r]^{Sq^2}
&
\Sigma^{-1} M
\ar[r]^{Sq^2}
&
\Sigma^{-2} M 
\ar[r]^(.4){Sq^3}
&
\Sigma^{-4}\annih_{(Sq^1)}M,
\ar[d]^{Sq^2 Sq^1 Sq^2}
\\
&&&&
\Sigma^{-8} \annih_{(Sq^1)}M
\ar[r]
&,
}
\end{equation}
where the morphisms are of degree $1$.

\begin{exam}
\label{exam:cochain-complex}
The cochain complex $\hom_{\cala (1)} (\fundcx, H\field^* (B\vd)) $  ($V$ an
elementary abelian $2$-group) is isomorphic to the complex obtained by applying
$[\Sigma^\infty B\vd, -]$ to the sequence of curved arrows of diagram
(\ref{diag:ko_post}). Hence, by the
techniques of Section \ref{sect:detect}, the homology of 
$\hom_{\cala(1)}(\fundcx, H\field^* (B\vd))$ is central to understanding  $V
\mapsto KO\langle n \rangle ^* (B\vd)$. 
\end{exam}

In applying the methods of Section \ref{sect:detect}, it is natural to reindex
in terms of the order of the spectra in the Postnikov tower, rather
than connectivity:

\begin{nota}
\label{nota:ko_reindex}
For an integer  $n= 4k +i$,  ($i, k\in \zed$ such that  
$0\leq i \leq 3$), write:
	\[
		KO\{ n \}:= \Sigma^{8k} KO\{ i \},
	\]
where
\[
	KO\{i \}  =
\left\{
\begin{array}{ll}
		ko\langle i \rangle & 0 \leq i \leq 2 \\
		ko \langle 4 \rangle & i=3.
\end{array}
\right.
\]
\end{nota}

\subsection{The complexification-realification sequences}
\label{subsect:eta-c-R}

Complex and orthogonal $K$-theories are related by 
the equivalence $KU \simeq KO \wedge C\eta$, which restricts to $ku \simeq ko
\wedge C\eta$ (cf. \cite{bg2}, for example). This yields the morphism between
the associated 
complexification-realification cofibre sequences:

\begin{equation}
\label{diag:eta-c-R}
 \xymatrix{
\Sigma ko 
\ar[r]^\eta
\ar[d]
&
ko 
\ar[r]^c
\ar[d]
&
ku \ar[r]^R
\ar[d]
&
\Sigma^2 ko
\ar[d]
\\
\Sigma KO 
\ar[r]_\eta 
&
KO 
\ar[r]_c 
&
KU
\ar[r]_R
&
\Sigma^2 KO.
}
\end{equation}

\begin{nota}
\label{nota:QO_ST}
(Cf. \cite{bg2}.)
Let $QO^*$ (respectively $ST^*$) denote the image (resp. kernel) of $ko^* (B\vd)
\rightarrow KO^* (B\vd)$. 
\end{nota}

There are natural short exact sequences (recall the notation of
\ref{nota:Q_TU}):
\begin{equation*}
\xymatrix @R=.25pc{
 0 \ar[r]& 
ST^* \ar[r]& 
ko^* (B\vd)\ar[r]&
 QO^* \ar[r] &0 \\
 0 \ar[r]& 
TU^* \ar[r]&
ku^* (B\vd)\ar[r]&
Q^* \ar[r]&
 0.
}
\end{equation*}

Hence, diagram (\ref{diag:eta-c-R}) induces a short exact sequence of complexes:
\begin{equation}
 \label{diag:ses_complexes}
 \xymatrix {
\ldots
\ar[r]
&
ST^{*+1} 
\ar[r]^\eta
\ar[d]
&
ST^*
\ar[r]^c
\ar[d]
&
TU^* 
\ar[r]^R
\ar[d]
&
ST^{*+2} 
\ar[r]
\ar[d]
&
\ldots
\\
\ldots
\ar[r]
&
ko^{*+1} (B\vd)
\ar[r]^\eta
\ar[d]
&
ko^* (B\vd)
\ar[r]^c
\ar[d]
&
ku^* (B\vd) 
\ar[r]^R
\ar[d]
&
ko^{*+2}(B\vd) 
\ar[r]
\ar[d]
&
\ldots
\\
\ldots
\ar[r]
&
QO^{*+1} 
\ar[r]_\eta 
&
QO^* 
\ar[r]_c
&
Q^* 
\ar[r]_R
&
QO^{*+2} 
\ar[r]
&
\ldots
}
\end{equation}
in which the middle complex is exact.

The top row can be considered as an exact couple, as in Appendix
\ref{sect:bockstein}; in particular, there is an associated Bockstein operator: 
$\bock^{*} : TU^{*}
\rightarrow
TU^{*+2}$. By Theorem \ref{thm:ku-detection}, $TU^*$ identifies explicitly as a
subfunctor of $H\field^* (B\vd)$.

\begin{prop} 
\cite{bg2}
\label{prop:bock_Sq2}
There is a natural commutative diagram:
\[
 \xymatrix{
TU^* 
\ar[r]_{\bock^*} 
\ar@{^(->}[d]
&
TU^{*+2} 
\ar@{^(->}[d]
\\
H\field^* (B\vd)
\ar[r]^{Sq^2}
&
H\field^{*+2} (B\vd).
}
\]
\end{prop}

\section{Cohomology of elementary abelian $2$-groups}
\label{sect:module}

The results of this section are formulated in the category of
bounded-below  $\cala(1)$-modules of finite type, which is
abelian, closed under tensor products and has  projective covers.

\begin{nota}
For $M$ an $\cala(1)$-module, let $\Omega M$ denote the first syzygy of $M$, 
namely the kernel of the surjection $P_M \twoheadrightarrow M$
from the projective cover of $M$. By convention, $\Omega^0 M = M$ and, for $n
\in \nat$, $\Omega^n M$ is defined by iteration.
\end{nota}

\begin{nota}
\ 
\begin{enumerate}
 \item 
 Let $P$ denote the reduced $\field$-cohomology of $B \zed/2$, which identifies
with the augmentation ideal
$\overline{\field_2 [u]}$ of $H\field^* (B\zed/2_+) \cong \field [u]$.
 \item
 (Cf. \cite[Section A.9]{bg2}.) Let $R$ denote the $\cala(1)$-module defined by
the non-split extension
$
 0 
\rightarrow 
P 
\rightarrow 
R
\rightarrow 
\Sigma^{-1} \field
\rightarrow 
0.
$
\item
Let $P_0$ denote the $\cala(1)$-module defined by the non-split extension  $0 
\rightarrow 
\field
\rightarrow 
P_0
\rightarrow 
R
\rightarrow 
0.
$
\end{enumerate}
\end{nota}

\begin{rem}
\ 
\begin{enumerate}
 \item 
  There is an isomorphism $P \cong \Sigma^{-1}\Omega P_0$ \cite{bruner_Ossa}.  
\item
$P$ is $Q_0$-acyclic and $R$ is $Q_1$-acyclic.
 \item
The modules $P_0$ and $\Sigma R$ are stably idempotent \cite{bruner_Ossa}. 
\end{enumerate}
\end{rem}

\begin{prop}
\label{prop:Pn_cohom_BV}
\cite{bruner_Ossa}
For $n\in \nat$, there is an isomorphism in the
category of
$\cala(1)$-modules: 
$
 P^{\otimes n+1} 
\cong 
F_n \oplus 
\Sigma^{-n}\Omega^n P, 
$
where $F_n$ is a free $\cala (1)$-module.
\end{prop}

\begin{proof}
 (Indications.) The proof is by induction upon $n$, starting with the case
$n=0$. 
It is clear that $P^{\otimes n}$ is $Q_0$-acyclic; hence,  by
the criterion for $\cala(1)$-freeness in terms of vanishing of Margolis homology
\cite{adams_margolis}, $P^{\otimes n} \otimes R$ is   $\cala(1)$-free. 
The result follows by considering the short exact sequence $P^{\otimes n}
\otimes (P \rightarrow R \rightarrow \Sigma^{-1} \field)$. 
 \end{proof}

 \begin{cor}
 \label{cor:Kunneth_decomposition}
  For $V$ an elementary abelian $2$-group of finite rank, there is a
(non-functorial) isomorphism of $\cala (1)$-modules:
  \[
   H\field^* (B \vd)
\cong 
F_V \oplus 
\bigoplus _{i \geq 1}
\Big(
\Lambda^{i} (V) \otimes \Sigma^{-i} \Omega^i P \Big),
  \]
where $\Lambda^{i} (V)$ is concentrated in degree zero and $F_V$ is a free
$\cala (1)$-module. 
 \end{cor}

 \begin{proof}
  This is a straightforward consequence of Proposition \ref{prop:Pn_cohom_BV}
and of the Künneth theorem applied to $B\vd
\simeq (B \zed/2) ^{\times \mathrm{rank}(V)}$.
 \end{proof}

\begin{rem}
 The functoriality with respect to $V$ can be analysed by introducing a
filtration and considering the associated graded. This is not required here,
since Lemma \ref{lem:exterior_powers} can be applied in the case of interest.
\end{rem}

\section{Functorial cohomology calculations}
\label{sect:calculate}

The abstract detection results of Section \ref{sect:detect}
are applied  to prove Proposition
\ref{prop:lower-bound}, which  gives a  lower bound for the
image of 
$$
	KO\langle n \rangle ^* (B\vd) \rightarrow KO^* (B\vd).
$$
 This relies upon calculating the cohomology of $
 \hom_{\cala(1)}(\fundcx, H\field^*(B\vd) ).
$
Since projective $\cala (1)$-modules are also injective (cf. \cite[Chapter 12,
Section 2]{margolis}), this  
reduces to the calculation of the cohomology of 
$
 \hom_{\cala(1)}(\fundcx, \Sigma^{-n} \Omega^{n} P), 
$
for $n \in \nat$, by Corollary  \ref{cor:Kunneth_decomposition}. This can be
reduced further to the calculation of the cohomology of
$\hom_{\cala(1)}(\fundcx,  P)$, by the following result.

\begin{prop}
\label{prop:Omega_décalage}
 Let $M$ be an $\cala (1)$-module which is bounded-below, of finite-type and
$Q_0$-acyclic. Then there is a natural isomorphism:
 \[
  H^n (\hom _{\cala (1)} (\fundcx , \Sigma^{-1} \Omega M) ) 
  \cong 
    H^{n-1} (\hom _{\cala (1)} (\fundcx , M) )
 \]
of graded vector spaces.
\end{prop}

\begin{proof}
Since $M$ is $Q_0$-acyclic, applying the functor $\hom _{\cala (1)} (\fundcx, -)
$ to the short exact
sequence $\Omega M \rightarrow P_M \rightarrow M $ 
yields an exact sequence of cochain  complexes
 \[
 0\rightarrow 
  \hom _{\cala (1)} (\fundcx, \Omega M)
 \rightarrow 
  \hom _{\cala (1)} (\fundcx, P_M)
  \rightarrow 
  \hom _{\cala (1)} (\fundcx, M)
  \rightarrow 
  0,
 \]
as can be seen as follows. The only non-projective terms of $\fundcx$ are
suspensions of $\cala (1)// \cala (0)$; since $\hom_{\cala (1)} (\cala(1) //
\cala (0), -)$ is
naturally isomorphic to $\hom_{\cala (0)} (\field, -) $, the fact that $\Omega M
\rightarrow
P_M \rightarrow M$ restricted to $\cala (0)$ splits (since $M$ is $Q_0$-free, by
hypothesis) implies the exactness.

The projective cover $P_M$ is also injective as an $\cala (1)$-module, thus the
middle complex is acyclic and  the associated
long exact sequence in cohomology provides the stated isomorphism. 
 The shift in degree corresponding to the $\Sigma^{-1}$ arises from the degree
of the morphisms in $\fundcx$.
\end{proof}

\begin{lem}
\label{lem:cohom_calc_P}
 The cohomology of $\hom_{\cala (1)} (\fundcx, P)$ has Poincaré series given 
by 
 \[
  \begin{array}{|l||l|}
 \hline 
 i & H^{4k+i} ( \hom_{\cala (1)} (\fundcx, P)) 
 \\
   \hline 
   \hline
 0 & t^{-8k} \big(\frac{t^{4} }{1- t^4}\big)
 \\
 \hline
 1 & t^{-8k} \big(\frac{1 }{1- t^4}\big)
 \\
 \hline
 2 & t^{-8k} \big(t^{-1} + \frac{1} {1-t^4} \big) 
 \\
 \hline 
 3 & t^{-8k} \big(t^{-2} + \frac{1} {1-t^4} \big) \\
 \hline
 \end{array}
 \]
 for integers $k$, $0 \leq i \leq 3$. In particular, in any given cohomological
and internal bidegree, the cohomology is at most one dimensional. 
\end{lem}

\begin{proof}
 By periodicity (up to suspension) of $\fundcx$, it suffices to calculate the
cohomology of the following cochain complex:
 \[
 \xymatrix{
\ar[r]&
 	\Sigma^4\overline{\field [u^2] }
 	\ar[d]_{Sq^2Sq^1 Sq^2}
 	\\
& 	
\overline{\field [u^2] }
 	\ar[r]^{Sq^2}
 	&
 	\Sigma^{-1} \overline{\field [u] }
 	\ar[r]^{Sq^2}
 	&
 	\Sigma^{-2}\overline{\field [u] }
 	\ar[r]^{Sq^3}
 	&
 	\Sigma^{-4} \overline{\field [u^2] }
 	\ar[d]^{Sq^2 Sq^1 Sq^2}
 	\\
 &	&&&
 	\Sigma^{-8} \overline{\field [u^2] }
\ar[r]
&.
}
 \]
The behaviour of the Steenrod operations on $u^n$ depends on the congruence
class  of $n$ modulo $4$; $Sq^2 (u^n)$ is non-zero if and only if $n \equiv 2, 3
\mod (4)$, $Sq^3 (u^n)$ is non-zero if and only if $n \equiv 3 \mod (4)$ and the
operation $Sq^2 Sq^1 Sq^2$ is identically zero on $\field [u^2]$. It follows
that the cohomology of the middle row is given by the classes:
\[
\begin{array}{|l||l|l|l|l|}
\hline
\mathrm{Cohom.\ degree}&
 0&1&2 &3\\
\hline
\hline
&
u^{4(k+1)}
&
\Sigma^{-1}u^{4k+1}
&
\Sigma^{-2}u, \ \Sigma^{-2}u^{4k+2}
&
\Sigma^{-4}u^2, \  \Sigma^{-4}u^{4(k+1)}
\\
\hline
\end{array}
\]
where $k \in \nat$. 
 \end{proof}

\begin{rem}
 The identification given in Example \ref{exam:cochain-complex} and the
application of the
detection arguments of Section \ref{sect:detect} imply that 
 the cohomological degree $n$ above corresponds to classes from $KO
\{n\}$-cohomology; this notation is adopted below.
\end{rem}

From this, the  cohomology 
$V \mapsto H_* (\hom _{\cala(1)} (\fundcx, H\field^*(B\vd)))$
can be deduced. The calculation is summarized in Proposition
\ref{prop:functorial_homology} and illustrated in Figure
\ref{table:functors}.
\begin{figure}[ht]
\[
	\begin{array}{|c||c|c|c|c|c|c|c|c|}
		\hline
	
		&
		KO\{-3\}
		&
		KO\{-2\}&
		KO\{-1 \}&
		KO\{0\}&
		KO\{ 1 \}&
		KO\{ 2 \}&
		KO\{ 3 \}&
		KO\{ 4 \}
		\\
		\hline\hline
		-4&
		&&&
		&
		&
		&
		&
		[p_1\Ibar]\\
		\hline
		-3&
		&&&
		&
		&
		&
		&
		\\
		\hline
		-2&
		&&&
		&
		&
		&
		[\Lambda^1]&
		[\Lambda^2]\\
		\hline
		-1&
		&&&
		&
		&
		[\Lambda^1]&
		[\Lambda^2]&
		[\Lambda^3]\\
		\hline
		0&
		&&&
		&
		[p_1 \Ibar]&
		[p_2\Ibar]&
		[p_3\Ibar]&
		[p_4 \Ibar]
		\\
		\hline
		1&
		&&&
		&
		&
		&
		&
		\\
		\hline
		2&
		&&&
		&
		&
		&
		&
		\\
		\hline
		3&
		&&&
		&
		&
		&
		&
		\\
		\hline
		4&
		&&&
		[p_1\Ibar]&
		[p_2\Ibar]&
		[p_3 \Ibar]&
		[p_4 \Ibar]&
		[p_5\Ibar]\\
		\hline
5&
&&&
		&
		&
		&
		&
		\\
		\hline
6&
&&
[\Lambda^1]&
		[\Lambda^2]&
		[\Lambda^3]&
		[\Lambda^4]&
		[\Lambda^5]&
[\Lambda^6]		
\\
		\hline
7&
&[\Lambda^1]&[\Lambda^2]&
		[\Lambda^3]&
		[\Lambda^4]&
		[\Lambda^5]&
[\Lambda^6]&
[\Lambda^7]		
\\
		\hline
8&
[p_1 \Ibar]
&[p_2 \Ibar]
&[p_3 \Ibar]
&
		[p_4 \Ibar]&
		[p_5\Ibar]&
[p_6\Ibar]&
		[p_7\Ibar]&
		[p_8 \Ibar]\\
		\hline
\mathrm{etc.}&\ldots&\ldots&\ldots&\ldots&\ldots&\ldots&\ldots&\ldots
\\
\hline
	\end{array}
\]
\caption{The Grothendieck group interpretation of the cohomology of $\hom_{\cala
(1)} (\fundcx, H\field^* (B\vd))$}
\label{table:functors} 
\end{figure}

\begin{prop}
\label{prop:functorial_homology}
The  non-zero values in the Grothendieck group of the functor  $V \mapsto H^*
(\hom _{\cala(1)} (\fundcx, H\field^*(B\vd)))$ are given in bidegree
$KO\{n\}^d$,
for $l \in \zed$,  by
\[
\begin{array}{|l||l|}
\hline
d & KO\{n\}\\
\hline\hline
8l & [p_{n+4l} \Ibar ]
\\
\hline
8l +4 & [p_{n+4l+1} \Ibar] \\
\hline
8l + 6
&
[\Lambda^{n+4l+1}] \\
\hline
8l+7 & [\Lambda^{n+4l+2}].
\\
\hline
\end{array}
\]
\end{prop}

\begin{proof}
The result follows from Lemma \ref{lem:cohom_calc_P}, Corollary
\ref{cor:Kunneth_decomposition} and Proposition \ref{prop:Omega_décalage}. 
For instance, the occurrence of the composition factors $\Lambda^1$ is given by
Lemma \ref{lem:cohom_calc_P}; the {\em décalage} provided by Proposition
\ref{prop:Omega_décalage} then 
 shows that each factor of $\Lambda^1$ gives rise to a factor of $\Lambda^2$ to
the right in Figure \ref{table:functors} and this pattern continues.

The proof that the result holds as a statement in the Grothendieck group is
a straightforward application of Lemma \ref{lem:exterior_powers}.
\end{proof}

\begin{defn}
\label{def:coker_image_QO}
For $n\in \zed$, define graded functors:
\begin{eqnarray*}
C\{n\}^*&:&V
\mapsto
\mathrm{Coker}
	\{
	KO\{n\}^* (B\vd) \rightarrow KO^* (B\vd)
	\}
	\\
\koimage \{n \}^* &:& V
\mapsto
\mathrm{Image}
	\{
	KO\{n\}^* (B\vd) \rightarrow KO^* (B\vd)
	\}.
\end{eqnarray*}
\end{defn}

In the notation of Proposition \ref{prop:detection_subquotient},  $\Phi_n [B\vd,
KO]^* = \koimage \{n\} ^*$; also $QO\{0 \}^* = QO^*$ of Notation
\ref{nota:QO_ST}.

\begin{lem}
\label{lem:filter_cokernel}
For $n \in \zed$, there is a natural short exact sequence
\[
	0
	\rightarrow
	\koimage\{n-1 \}^* / \koimage\{n\}^*
	\rightarrow
	C\{n \}^*
	\rightarrow C\{n-1 \}^*
	\rightarrow
0
\]
and, in a fixed degree $d$, $C\{n \}^d$ admits a finite filtration with
associated graded
\[
	\bigoplus_{j <n}
	\koimage \{ j\}^d / \koimage\{ j+1 \} ^d .
\]
\end{lem}

\begin{proof}
	By definition, there is a short exact sequence of graded functors
	\[
		0
		\rightarrow
		\koimage\{n \}^*
		\rightarrow
		KO^* (B\vd)
		\rightarrow
		C\{n \}^*
		\rightarrow
		0.
	\]
The inclusion $\koimage\{ n \}^* \hookrightarrow \koimage^* \{ n-1 \} $ induces
the stated short exact sequence. The second statement follows recursively, using
the observation that, in a fixed degree $d$,  $C\{n \}^d
= 0$ for $n \ll 0 $.
\end{proof}

\begin{prop}
\label{prop:lower-bound}
For $n,d\in \zed$, $C\{n\} ^d $ is a finite functor.
Moreover, there are inequalities in the Grothendieck group:
\[
	[C\{n \}^d ]
	\leq
	\left\{
\begin{array}{lll}
\sum_{s=1}^{4l+n-1} [p_s\Ibar] &=\ [\Pbarzedtwo / \Pbarzedtwo^{4l+n}] & d =8l\\
\sum_{s=1} ^{4l+n\phantom{-1}}[p_s\Ibar]&=\ [\Pbarzedtwo / \Pbarzedtwo^{4l+n+1}]
 &  d= 8l +4 \\
\sum_{s=1} ^{4l+n+1}[\Lambda^s]&=\ [\Pbar / \Pbar^{4l+n+2}] & d =8l +6 \\
\sum_{s=1} ^{4l+n+2}[\Lambda^s]&=\ [\Pbar / \Pbar^{4l+n+3}] &  d =8l +7
\end{array}
\right.
\]
and, in the remaining cases, $C\{n\} ^d =0$.

In a  fixed degree $d$, equality holds if and only if, for all $j <n$:
\[
 \koimage \{ j\}^d / \koimage\{ j+1 \} ^d 
\cong 
\mathrm{Ker}(\theta_j )^d /
\mathrm{Im} (\Sigma^{-1} \theta_{j-1} )^d.
\]
\end{prop}

\begin{proof}
The stated equalities in the Grothendieck group follow from
Proposition \ref{prop:Grothendieck-group} and Example
\ref{exam:grothendieck_gp}.

Lemma \ref{lem:filter_cokernel} gives 
$
[C\{n \}^d]=  \sum_{j <n}
[
	\koimage \{ j\}^d / \koimage\{ j+1 \} ^d 
],
$
hence, to prove the inequality,  it suffices to give an upper bound for $
[
	\koimage \{ j\}^d / \koimage\{ j+1 \} ^d 
]$; this is provided by  Lemma \ref{lem:subquotient_Grothendieck}.

Proposition \ref{prop:detection_subquotient} implies that 
$\koimage \{ j\}^d / \koimage\{ j+1 \} ^d $ is a subquotient of
$$
\mathrm{Ker}(\theta_j )^d / \mathrm{Im} (\Sigma^{-1} \theta_{j-1} )^d
$$
and the value of the  latter in the Grothendieck
group is given by   Proposition \ref{prop:functorial_homology}; this proves
the inequalities.

Finally, since the functors involved are finite, equality holds in degree
$ d$  if and only $ \koimage \{ j\}^d / \koimage\{ j+1 \} ^d
\cong 
\mathrm{Ker}(\theta_j )^d /
\mathrm{Im} (\Sigma^{-1} \theta_{j-1} )^d$ for all $j < n$.
\end{proof}

\section{A $Sq^2$-homology calculation}
\label{sect:sqhom}

Recall that $TU^*$ identifies as the image of the iterated Milnor operation
$Q_0Q_1 : H\field^{*-4} (B\vd) \rightarrow H\field^{*}
(B\vd)$. Proposition \ref{prop:bock_Sq2} implies that the operation $Sq^2$
induces a complex
\[
	\ldots \rightarrow TU^{*-2}\stackrel{Sq^2}{\rightarrow} TU^{*}
\stackrel{Sq^2}{\rightarrow} TU^{*+2}\rightarrow \ldots .
\]
The work of Bruner and Greenlees \cite{bg2} on the $ko$-(co)homology of
elementary abelian $2$-groups  shows the importance of the calculation of the
homology of this complex.  In 
\cite[Proposition 9.7.2]{bg2}, they calculate the homology and their result can
be interpreted as a functorial calculation.

The purpose of this section is to show that the methods employed in Section
\ref{sect:calculate} provide an alternative, direct proof. However, it is no
longer possible to reduce to a calculation involving 
only $P$, since there is no analogue of Proposition \ref{prop:Omega_décalage} in
this case. Thus further precision is required on the structure 
of the $\cala(1)$-modules $\Sigma^{-n} \Omega^n P$. 

\begin{nota}
\ 
\begin{enumerate}
 \item 
 For $n \in \nat$, let $P_n$ denote the $\cala (1)$-module
$\Sigma^{-n}\Omega^{n} P_0 $, so that $P_1 = P$.
\item
Let $J$ denote the  $\cala(1)$-module $\Sigma^{-2} \cala (1)/
(Sq^1Sq^2)$.
\end{enumerate}
\end{nota}

The module $J$ is an element of the stable Picard group of $\cala (1)$-modules
of order $2$, namely $J  \otimes J \cong \field \oplus F$, where $F$ is free
(cf. \cite{adams_priddy}).

\begin{thm}
\cite{bruner_Ossa,yu}
\label{thm:structure_Pi}
For $n \in \nat$, there is an isomorphism of $\cala (1)$-modules 
$ 
 P_{n+4} \cong \Sigma^8 P_n.
$ 
\end{thm}

\begin{rem}
\label{rem:tensor-J}
 The periodicity can be seen by establishing that there is an isomorphism of
$\cala(1)$-modules 
 $
  P_0 \otimes J \cong \Sigma^{-4} P_2 \oplus F',
 $ 
where $F'$ is free. It follows also that $P \otimes J \cong \Sigma^{-4} P_3
\oplus F''$ for a free $\cala(1)$-module $F''$.
\end{rem}

\begin{prop}
\label{prop:sq_cohom}
	For $n \in \zed$,
	\[
		\mathrm{Ker} \{ Sq^2 : TU^n \rightarrow TU^{n+2} \}/
		\mathrm{Im} \{Sq^2: TU^{n-2} \rightarrow TU^n \}
		\cong
		\left\{
		\begin{array}{ll}
			\Lambda^{4k+2} & n= 8k +6 \\
			\Lambda^{4k+3} & n= 8k +7 \\
			0 & \mathrm{otherwise}.
		\end{array}
		\right.
	\]
\end{prop}

\begin{proof}
Consider the isomorphism 
	$
H\field^* (B \vd)
\cong
F_V \oplus 
\bigoplus _{i\geq 1}
\Lambda^i (V) \otimes P_i
$ in $\cala(1)$-modules, 
where $F_V$ is a free $\cala(1)$-module (bounded-below, of finite type); 
 the $Sq^2$-complex splits as a corresponding direct sum. Using the periodicity
isomorphism for the $P_i$'s given by Theorem \ref{thm:structure_Pi}, this
reduces the calculation of the $Sq^2$-homology evaluated upon $\vd$ to
the calculation of the respective homologies for the 
$\cala (1)$-modules:
 $ 
 	\cala(1) , P_0, P_1, P_2, P_3.
 $
 \begin{enumerate}
 \item
 The image of $Q_0Q_1$ applied to $\cala (1)$ has two classes, which are linked
by the operation $Sq^2$, hence the free summand contributes nothing to
the homology.
 \item
 By inspection, the operation $Q_0Q_1$ acts trivially upon $P_0$ and $P_1$,
hence these contribute nothing to the $Sq^2$-homology.
 \item
 The structure of $P_2$, $P_3$ is described explicitly in \cite{bruner_Ossa,yu}.
The relevant part of the structure can be understood using Remark
\ref{rem:tensor-J}; the non-trivial morphism  $\field \hookrightarrow P_0$
induces an embedding $\Sigma^4 J \hookrightarrow P_2$ and  the surjection $P
\rightarrow \Sigma \field$ induces a surjection $P_3 \twoheadrightarrow \Sigma^5
J$. Upon restricting to the subalgebra $E(1):= \Lambda (Q_0, Q_1) \subset \cala
(1)$, this gives isomorphisms
\begin{eqnarray*}
 P_2 |_{E(1) } &\cong& \Sigma ^2 E(1) \oplus L_2\\
  P_3|_{E(1) } &\cong& \Sigma ^3 E(1) \oplus L_3
\end{eqnarray*}
where $L_2$, $L_3$ are indecomposable $E(1)$-modules on which $Q_0Q_1$ acts
trivially.  
 
Thus, for both $P_2$ and $P_3$, the image of $Q_0Q_1$ is a single class and 
$Sq^2$ acts trivially on the associated complex. Explicitly, for $P_2$, the
$Sq^2$-homology is $1$-dimensional,
concentrated in degree $6$ and, for $P_3$, is $1$-dimensional, concentrated in
degree $7$. 
 \end{enumerate}
Lemma \ref{lem:exterior_powers} implies that these classes
correspond to the simple functors $\Lambda^2$ and $\Lambda^3$ in degrees $6$ and
$7$ respectively. The general result follows, by using the periodicity  given by
Theorem \ref{thm:structure_Pi}.
\end{proof}

\section{Detection for $ko$}
\label{sect:cohom}

This section  determines the functorial structure of $ko^*
(B\vd)$ as a first step towards the determination of $KO\langle n \rangle ^*
(B\vd)$; the arguments use the abstract detection result,
Proposition \ref{prop:detection_subquotient}, which depends upon understanding
the image of $KO\langle n \rangle ^* (B\vd) \rightarrow KO^* (B\vd)$.

In degrees which are multiples of four, a direct approach treating all the cases
simultaneously is possible, using the fact
that $KO^* (B\vd) \rightarrow KU^* (B\vd)$ is injective in these degrees, so
that the known structure of $ku^* (B\vd)$ can be used to provide an  upper
bound for the image of
$KO\langle n \rangle ^* (B\vd)$, which can be played off
against the lower bound provided by Proposition \ref{prop:lower-bound}. In the
remaining degrees in which $KO^* (B\vd)$ is non-trivial (those congruent
to $6, 7 \mod 8$), the map to $KU^* (B\vd)$ is zero, hence this strategy cannot
be applied. Instead,  a Bockstein argument derived from 
 the complexification-realification cofibre sequence of Section 
\ref{subsect:eta-c-R} is used.

The $KO$-cohomology $KO^*(B\vd)$  can be deduced from the
case of $KU$, which is concentrated in even degrees, where $KU^{2d} (B\vd) \cong
\Pbarzedtwo (V)$ (see \cite{powell} and compare Theorem \ref{thm:ku-detection}),
by using the long
exact sequence associated to  $\Sigma KO \stackrel{\eta}{\rightarrow} KO
\stackrel{c}{\rightarrow} KU \stackrel{R}{\rightarrow} \Sigma^2 KO.$ 

\begin{prop}
\label{prop:KO_cohom_BV}
(Cf. \cite{bg2}.)
	There are isomorphisms
\[
	KO^{8k+l} (B\vd) \cong
\left\{
\begin{array}{ll}
\Pbar & l=7\\
\Pbar & l=6 \\
\Pbarzedtwo & l=4 \\
\Pbarzedtwo & l=0\\
0 & \mathrm{otherwise}. 
\end{array}
\right.
\]
\begin{enumerate}
 \item 
Complexification $c: KO ^{8k+l} (B\vd) \rightarrow KU^{8k+l} (B\vd)$ is zero
unless $l \equiv 0 \mod 4$; for  $l=0$ it is an isomorphism and, for $l=4$, 
$\Pbarzedtwo \stackrel{2}{\hookrightarrow} \Pbarzedtwo.$
\item
Realification $R: KU^{8k+l-2} (B\vd) \rightarrow KO^{8k+l} (B\vd)$ is zero
unless $l \in \{ 0,4,6\}$; for $l=0$ it is $
 \Pbarzedtwo \stackrel{2}{\hookrightarrow} \Pbarzedtwo$, for $l=4$ is an
isomorphism and, for $l=6$ is the surjection 
$\Pbarzedtwo \twoheadrightarrow \Pbar$.
\item
$KO^* (B\vd)\stackrel{\eta}{ \rightarrow} KO^{*-1} (B\vd)$,  is
zero except for 
\[
\xymatrix{
	KO^{8(k+1)} (B\vd) \cong \Pbarzedtwo
\ar@{->>}[r]^\eta
&
KO ^{8k+7}
(B\vd)  \cong \Pbar 
\ar[r]^\eta_\cong
&
KO^{8k+6}
(B\vd) \cong \Pbar.
}
\]
\end{enumerate}
\end{prop}

The key to the calculation of $ko^* (B\vd)$ is the short exact sequence of
complexes (\ref{diag:ses_complexes}) of Section \ref{subsect:eta-c-R}. Recall
the functors $Q^*, QO^*$  of Notations \ref{nota:Q_TU} and
\ref{nota:QO_ST}.

\begin{prop}
\label{prop:QO}
 For $k \in \zed$, there are isomorphisms:
\[
 QO^{8k+l} 
\cong 
\left\{
\begin{array}{ll}
 \Pbar^{4k+3} & l=7\\
\Pbar^{4k+2} & l=6\\
\Pbarzedtwo^{4k+1} & l=4\\
\Pbarzedtwo^{4k} & l=0\\
0 & \mathrm{otherwise}.
\end{array}
\right.
\]
Moreover, complexification $QO^{8k+l} \stackrel{c}{\rightarrow} Q^{8k+l} $ is
zero unless $l
\equiv 0 \mod 4$; $QO^{8k} \stackrel{c}{\rightarrow}  Q^{8k}$ is an isomorphism
and $QO^{8k+4}
\stackrel{c}{\rightarrow}  Q^{8k+4}$ is the inclusion $\Pbarzedtwo^{4k+1}
\hookrightarrow
\Pbarzedtwo^{4k+2}$. 

The complex $\ldots \rightarrow QO^{*+1} \stackrel{\eta}{\rightarrow} QO^*
\stackrel{c}{\rightarrow} Q^* \stackrel{R}{\rightarrow} \ldots$ is exact, except
for the segments:

\[
 \xymatrix{
QO^{8k+7} 
\ar@{=}[d]
\ar[r]^\eta
&
QO^{8k+6} 
\ar@{=}[d]
\ar[r]^c
&
Q^{8k+6} 
\ar@{=}[d]
\ar[r]^R
&
QO^{8(k+1)} 
\ar@{=}[d]
\ar[r]^\eta
&
QO^{8k+7} 
\ar@{=}[d]
\ar[r]^c
&
Q^{8k+7}
\ar@{=}[d]
\\
\Pbar ^{4k+3} 
\ar@{^(->}[r]
&
\Pbar ^{4k+2} 
\ar[r]_0
\ar@{.>>}[d]
&
\Pbarzedtwo ^{4k+3} 
\ar@{^(->}[r]
&
\Pbarzedtwo ^{4k+4} 
\ar[r]
&
\Pbar ^{4k+3} 
\ar[r]_0
\ar@{.>>}[d]
&
0\\
0
&
\Lambda^{4k+2}
&
0
&
0
&
\Lambda^{4k+3}
&
0
}
\]
where the homology is given by the  bottom line, with the corresponding
surjections indicated by the dotted arrows.
\end{prop}

\begin{proof}
The morphism $KO^* (B\vd) \stackrel{c}{\rightarrow}  KU^* (B\vd) $ induces an
inclusion
\[
	QO^{8k + 4 \epsilon } \hookrightarrow QU^{8 k+ 4 \epsilon } \cong
\Pbarzedtwo^{4k+2\epsilon},
\]
for $\epsilon \in \{0, 1\}$. This gives
\begin{eqnarray*}
	QO^{8k} &\subseteq& \Pbarzedtwo ^{4k} \\
	QO^{8k +4} & \subseteq & \Pbarzedtwo^{4k+1} \cong (2\Pbarzedtwo) \cap
\Pbarzedtwo^{4k+2}
\end{eqnarray*}
as upper bounds and the inclusions $QO^{8k+4\epsilon} \hookrightarrow
KO^{8k+4\epsilon}(B\vd)$ correspond respectively to
\begin{eqnarray*}
	&&QO^{8k} \hookrightarrow  \Pbarzedtwo ^{4k} \hookrightarrow \Pbarzedtwo
\\
	&&QO^{8k +4} \hookrightarrow \Pbarzedtwo^{4k+1} \hookrightarrow
\Pbarzedtwo.
\end{eqnarray*}

A comparison between the cokernels of $\Pbarzedtwo^{4k}\hookrightarrow
\Pbarzedtwo$ (respectively $ \Pbarzedtwo^{4k+1} \hookrightarrow \Pbarzedtwo$)
and the bounds provided by Proposition \ref{prop:lower-bound} shows that the
inequalities are isomorphisms, by Proposition \ref{prop:Grothendieck-group}.

In the remaining non-trivial cases, in degrees congruent to $6, 7 \mod 8$, an
upper bound is obtained by appealing to the general method of Appendix
\ref{sect:bockstein}, as follows.

Multiplication by $\eta$ gives the commutative diagram
\[
 \xymatrix{
QO^{8(k+1)}
\ar[rr]^\eta
\ar[d]_\cong
&&
QO^{8k+7}
\ar@{^(->}[d]
\\
\Pbarzedtwo^{4(k+1)}
\ar@{^(->}[r]
&
\Pbarzedtwo
\ar@{->>}[r]
&
\Pbar,
}
\]
which identifies the image of $QO^{8(k+1)}$ in $QO^{8k+7}$ as $\Pbar^{4(k+1)}$.

The complexes 
\begin{eqnarray*}
&&QO^{8(k+1)} \stackrel{\eta} {\rightarrow} QO ^{8k+7} \rightarrow Q^{8k+7} =0
\\
&&QO^{8k+7} \stackrel{\eta} {\rightarrow} QO ^{8k+6} \stackrel{0}{\rightarrow}
Q^{8k+6}
\end{eqnarray*}
(where the last morphism is zero, since $Q^{8k+6}$ takes values in torsion-free
abelian groups and $QO^{8k+6}$ is torsion) have homology appearing as a
subquotient of  the simple functors 
$\Lambda^{4k+3}$ and $\Lambda^{4k+2}$ respectively, by Lemma
\ref{lem:subquotient}, using Proposition \ref{prop:sq_cohom} and the shift in
homological degrees associated to the short exact sequence of complexes
(\ref{diag:ses_complexes}) of Section \ref{subsect:eta-c-R}. This provides the 
upper bounds:
\begin{eqnarray}
\label{eqn:inequal67}
	QO^{8k+7} &\subseteq& \Pbar^{4k+3}\\
QO^{8k+6} &\subseteq& \Pbar^{4k+2}, \notag	
\end{eqnarray}
where both are equalities if $QO^{8k+6}=\Pbar^{4k+2}$.

Realification
\[
 \xymatrix{
\Pbarzedtwo^{4k+2} \cong Q^{8k+4} 
\ar[r]^R
\ar@{^(->}[d]
&
QO^{8k+6} 
\ar@{^(->}[d]
\\
\Pbarzedtwo\cong KU^{8k+4} (B\vd) 
\ar@{->>}[r]
&
KO^{8k+6} (B\vd) \cong\Pbar
}
\]
gives a lower bound of $\Pbar^{4k+2}$ for $QO^{8k+6}$, whence it follows that
both
the inequalities in (\ref{eqn:inequal67}) are equalities. 

Finally, using the structure of the functors $\Pbar^t$ and $\Pbarzedtwo^t$ (as
reviewed in Section \ref{sect:functors}), it is straightforward to calculate the
homology of the complex $\ldots \rightarrow QO^{*+1}
\stackrel{\eta}{\rightarrow} QO^* \stackrel{c}{\rightarrow} Q^*
\stackrel{R}{\rightarrow} \ldots$.
\end{proof}

\begin{cor}
\label{cor:ko_cohom}
Detection holds for  $ko$-cohomology of elementary abelian $2$-groups:
 the morphisms $ko \rightarrow KO$ and $ko \rightarrow H\zed$ induce a
natural monomorphism 
\[
	ko^* (B\vd) \hookrightarrow H\zed^* (B\vd) \oplus KO^* (B\vd).
\]
The functor $ST^*$ is the image of  $Sq^2 : TU^{*-2}
\rightarrow
TU^*$ and the morphism $\eta : ST^{*+1} \rightarrow ST^*$ is trivial.
\end{cor}

\begin{proof}
By applying the long exact sequence in homology associated to the short exact
sequence of complexes (\ref{diag:ses_complexes}) of Section
\ref{subsect:eta-c-R}, Proposition \ref{prop:QO} implies that the exact couple
$\ldots
\rightarrow ST^{*+1} \rightarrow ST^* \rightarrow TU^* \rightarrow \ldots$ has
homology concentrated at the
$TU^*$ term, where it coincides with the Bockstein homology. 

Therefore Proposition \ref{prop:eta-torsion} applies; it follows that
$ST^* \rightarrow TU^*$ is a monomorphism and that $ST^*$ is the
image of the operator $TU^{*-2}\rightarrow TU^*$, which is induced by $Sq^2$, by
Proposition \ref{prop:bock_Sq2}.

To show detection for $ko$, it suffices to show that $ST^*$ maps monomorphically
to $H\zed^* (B\vd)$. 
By the above,  $ko^*(B\vd)\rightarrow
ku^* (B\vd)$ induces an injection
$ST^* \hookrightarrow TU^*$, and  the composite $TU^* \rightarrow
ku^* (B\vd)\rightarrow H\zed^*(B\vd)$ is a monomorphism, by detection for $ku$
(Theorem \ref{thm:ku-detection}), hence the result follows.
\end{proof}

\section{Detection for $KO\langle n \rangle$}
\label{sect:postnikov_detect}

Throughout this section, the reindexing of the spectra $KO \langle n \rangle$
introduced in
Notation \ref{nota:ko_reindex} is used; for example, as in Definition
\ref{def:coker_image_QO}, $\koimage\{n\}$ is the image of
 $KO\{n \}(B\vd)$ in $KO (B\vd)$. Similarly, $\theta_n$ denotes the stable
cohomology operation derived
from the Postnikov tower of $KO$, as in Section \ref{sect:detect}.

\begin{thm}
	\label{thm:postnikov_detect}
	For each $n \in \zed$,  detection of level $n$ with respect to the
family of spectra $\{ \Sigma^\infty B (\zed/2)^{\oplus d} | 1\leq d \in \zed \}$
holds for the Postnikov  tower $KO \{n \}$. 
\end{thm}

\begin{proof}
	The result follows from the general result on detection,
Proposition \ref{prop:detection_subquotient}. Using the notation of {\em loc.
cit.}, the functorial homology $\mathrm{Ker}(\theta_n)/ \mathrm{Im}
(\Sigma^{-1} \theta_{n-1}) $ is a finite functor in each degree, hence to prove
weak detection at each level, it is sufficient to show that the filtration
quotient $\Phi_n [B\vd, KO]^* / \Phi_{n-1} [B\vd,KO]^*$ is abstractly isomorphic
to
$\mathrm{Ker}(\theta_n)/ \mathrm{Im} (\Sigma^{-1} \theta_{n-1}) $, for each $n$.
Here, by definition $\Phi_n  [ B\vd, KO]^*$ is the graded functor $\koimage
\{n\}$.

Proposition \ref{prop:QO}, establishes that the inequalities of Proposition
\ref{prop:lower-bound} are  equalities for $n=4m$, $m \in \zed$. To conclude, one argues as in  
Proposition \ref{prop:lower-bound}:
  Lemma \ref{lem:filter_cokernel} provides the equality
  \begin{eqnarray}
  \label{eqn:inequality}
  \nonumber
   \big[\ker \big\{ C\{4(m+1)\}^* \twoheadrightarrow C\{4m\}^* \} \big] & = & \sum_{j=4m}^{4m+3} [\koimage \{ j\}^* / 
\koimage\{ j+1 \} ^*]
   \\
   &\leq & \sum_{j=4m}^{4m+3} [\mathrm{Ker}(\theta_j )^* /\mathrm{Im} 
(\Sigma^{-1} \theta_{j-1} )^*]
   \end{eqnarray}
and Proposition \ref{prop:detection_subquotient}  the inequality. As explained 
above, the left hand side is determined by Proposition \ref{prop:lower-bound} 
and coincides with the lower term (\ref{eqn:inequality}), by Proposition 
\ref{prop:functorial_homology}; the functors occurring are finite in each degree, so the inequality is in fact an equality.
Hence, the final statement of Proposition \ref{prop:lower-bound} provides the 
required isomorphism, thus proving weak detection.

Finally, Lemma \ref{lem:weak-strong}  establishes detection at each level, since
detection has been proved for $ko$, by Corollary \ref{cor:ko_cohom}, hence holds
by Bott
periodicity for all of the theories $KO\{ 4s \} $, $s \in \zed$.
\end{proof}

From this one derives the explicit description of the functors $KO\{n \}
^* (B\vd)$,  in particular recovering the  results of \cite{bg2} for $ko$.

\begin{cor}
\label{cor:ses_kon}
 For $n \in \zed$, there is a natural short exact sequence:
\[
 0
\rightarrow 
\mathrm{Im} (\Sigma^{-1}\theta_{n-1})
\rightarrow 
KO\{n \} ^* (B\vd)
\rightarrow 
QO\{n \}
\rightarrow 
0,
\]
which is determined as a pullback of the short exact sequence 
associated to the quotient $\mathrm{Ker}(\theta_n) / \mathrm{Image}
(\Sigma^{-1} \theta_{n-1}) $, by Theorem  \ref{thm:functorial_ses_detection}.

The non-zero functors $QO\{i\}^{8k +l} $, for $0 \leq i \leq 3$ and $0 \leq l
\leq 7$, are given by:
\[
 \begin{array}{|l||l|l|l|l|}
  \hline
i& 8k & 8k +4 & 8k +6 & 8k +7 \\
\hline
\hline
0 & \Pbarzedtwo^{4k}&
\Pbarzedtwo^{4k+1}& 
\Pbar^{4k+2}&
\Pbar^{4k+3}
\\
1 & \Pbarzedtwo^{4k+1}& 
 \Pbarzedtwo^{4k+2} & 
\Pbar^{4k+3}&
\Pbar^{4(k+1)}
\\
2 & \Pbarzedtwo^{4k+2} & 
\Pbarzedtwo^{4k+3} &
\Pbar^{4(k+1)}&
\Pbar^{4(k+1)+1}
\\
3 & \Pbarzedtwo^{4k+3} &
\Pbarzedtwo^{4(k+1)} &
\Pbar^{4(k+1)+1}&
\Pbar^{4(k+1)+2}
\\
\hline
 \end{array}
\]
which determines the functors $QO\{n \}$, $\forall n \in \zed$, by Bott
periodicity.

The subfunctors $\mathrm{Image}(\Sigma^{-1} \theta_n) $ are given for $0 \leq n
\leq 3$ by:
\[
	\begin{array}{|l||l|}
\hline
n &\mathrm{Image}(\Sigma^{-1} \theta_n)\\
\hline
\hline
0& \mathrm{Image} \{H\zed^{*-5} (B\vd)
\stackrel{Sq^2Sq^1Sq^2}{\rightarrow}
H\zed^* (B\vd) \}
\}
\\
&
\cong \mathrm{Image} \{H\field^{*-6} (B\vd)
\stackrel{Sq^2Sq^2 Sq^2}{\rightarrow}
H\field^* (B\vd) \}\\
\hline
1 &\mathrm{Image} \{H\zed^{*-1} (B\vd)
\stackrel{Sq^2}{\rightarrow}
H\zed^{*+1} (B\vd) \}
\}
\\
&
\cong \mathrm{Image} \{H\field^{*-2} (B\vd)
\stackrel{Sq^2Sq^1}{\rightarrow}
H\field^{*+1} (B\vd) \}\\
\hline
2 &  \mathrm{Image} \{H\field^{*} (B\vd)
\stackrel{Sq^2}{\rightarrow}
H\field^{*+2} (B\vd) \}
\\
\hline
3 & \mathrm{Image} \{H\field^{*+1} (B\vd)
\stackrel{Sq^3}{\rightarrow}
H\field^{*+4} (B\vd) \}
\\
\hline
	\end{array}
\]
which extends to all integers $n$ by Bott periodicity.
\end{cor}

\begin{proof}
 The short exact sequence is provided by Proposition
\ref{prop:detection_subquotient} and Theorem 
\ref{thm:functorial_ses_detection}, as a consequence of detection 
established in
Theorem \ref{thm:postnikov_detect}. 

The identification of the functors $QO\{i
\}$ is a straightforward consequence of the equalities derived from Proposition
\ref{prop:lower-bound} in the proof of Theorem \ref{thm:postnikov_detect}
above, using the  structure of the
functors $\Pbarzedtwo^t$  reviewed in Section \ref{sect:functors}.
\end{proof}

\appendix
\section{General Bockstein results }
\label{sect:bockstein}

Fix an exact couple in an abelian category, considered as a complex of the form
\[
	\ldots \rightarrow \dcouple^{n+1} \stackrel{i^{n+1}}{\rightarrow} 
\dcouple^n
\stackrel{q^n}{\rightarrow} \ecouple^n \stackrel{\partial^n} {\rightarrow}
\dcouple^{n+2}
\rightarrow \ldots
 .
\]

The associated Bockstein-type operator (the differential associated to the exact
couple)
is $\bock^n : \ecouple^n \rightarrow \ecouple^{n+2}$, defined by $\bock^n :=
q^{n+2} \circ
\partial^n $.

The following is clear:

\begin{lem}
\label{lem:subquotient}
	For $n \in \zed$,
	\[
		\mathrm{Im} (\bock^{n-2}) \subseteq \mathrm{Im}(q^n) \subseteq
\mathrm{Ker} (\partial^n) \subseteq \mathrm{Ker} (\bock^n), 
	\]
hence 
$H^n:= \mathrm{Ker} (\partial^n ) /\mathrm{Im} (q^{n})
$ is a subquotient of $H_{\bock}^n := \mathrm{Ker} (\bock^n) / \mathrm{Im}
(\bock^{n-2}) $.

Moreover if $H_{\bock}^n$ has a finite composition series, 
$ H^n \cong H^n_\bock$  if and only if $\mathrm{Im} (\bock^{n-2} ) =
\mathrm{Im} (q^{n} ) $ and $\mathrm{Ker} (\bock^n) = \mathrm{Ker} (\partial^n)$.
\end{lem}

This is applied in the following basic result.

\begin{prop}
\label{prop:eta-torsion}
	Suppose that the exact couple $\dcouple^{*+1} \stackrel{i}{\rightarrow}
\dcouple^*
\stackrel{q}{\rightarrow} \ecouple^*
\stackrel{\partial}{\rightarrow} \dcouple^{*+2} $ satisfies the following
hypotheses:
\begin{enumerate}
	\item
$\dcouple^n = 0$ for $n \ll 0$;
\item
the complex is exact except at the terms $\ecouple^n$, where the homology $H^n$
coincides with $H_\bock^n$;
\item
$H_\bock^n$ has a finite composition series, $\forall n \in \zed$.
\end{enumerate}
Then $i^n =0$ for all $ n \in \zed$ and the complex decomposes as complexes
of the form:
\[
	\dcouple^n \hookrightarrow \ecouple^n \twoheadrightarrow \dcouple^{n+2}.
\]
In particular, $\dcouple^n$ identifies with the image of the  operator
$\bock^{n-2}$.
\end{prop}

\begin{proof}
	The result follows by an increasing induction upon $n$, using the
hypothesis $\dcouple^n=0$ for $n \ll 0$ for the initial step.

Suppose that $i^{n} :\dcouple^{n} \rightarrow \dcouple^{n-1}$ is zero. Exactness
of
$\ecouple^{n-2}\stackrel{\partial^{n-2}}{\rightarrow} \dcouple^{n}
\stackrel{i^{n}}{\rightarrow} \dcouple^{n-1}$ implies that $\partial^{n-2}$ is
an
epimorphism; the hypothesis  $H_\bock^n = H^n$ gives $\mathrm{Ker} (\bock^{n-2}
) = \mathrm{Ker} (\partial^{n-2} ) $, by Lemma \ref{lem:subquotient}.

Using this fact, inspection of 
\[
 \xymatrix{
&
\ecouple^{n-2}
\ar[rd]^{\bock^{n-2}}
\ar@{->>}[d]_{\partial^{n-2}}
\\
\dcouple^{n+1}
\ar[r]_{i^{n+1}}
&
\dcouple^n 
\ar[r]_{q^n} 
&
\ecouple^n
}
\]
shows that $q^{n}$ is a monomorphism and exactness of $ \dcouple^{n+1}
\stackrel{i^{n+1}}{\rightarrow} \dcouple^{n}\stackrel{q^n}{\rightarrow}T^n$
implies
that $i^{n+1} : \dcouple^{n+1} \rightarrow \dcouple^{n}$ is zero. Finally, the
above
identifies the image of $\dcouple^n$ in $\ecouple^n$ with the image of
$\bock^{n-2}$. This
completes the inductive step.
\end{proof}

\begin{rem}
The proof only requires that $\mathrm{Ker} (\bock^{n-2} ) = \mathrm{Ker}
(\partial^{n-2})$; for the application, the equivalent homological formulation
is convenient.
\end{rem}



\providecommand{\bysame}{\leavevmode\hbox to3em{\hrulefill}\thinspace}
\providecommand{\MR}{\relax\ifhmode\unskip\space\fi MR }
\providecommand{\MRhref}[2]{%
  \href{http://www.ams.org/mathscinet-getitem?mr=#1}{#2}
}
\providecommand{\href}[2]{#2}

\end{document}